\documentclass[aos]{imsart}
\RequirePackage[OT1]{fontenc}
\RequirePackage{amsthm,amsmath}
\RequirePackage[colorlinks,citecolor=blue,urlcolor=blue]{hyperref}

\usepackage[latin1]{inputenc}
\usepackage{lipsum}
\usepackage{amsmath}
\usepackage{bm}
\usepackage{amsfonts}
\usepackage{amsthm}
\usepackage{cleveref}
\usepackage[shortlabels]{enumitem}
\usepackage{txfonts}
\usepackage{stmaryrd}
\usepackage{amssymb}
\usepackage{makeidx}
\usepackage{appendix}
\usepackage{afterpage}
\usepackage{mathrsfs}
\usepackage{natbib}
\usepackage{relsize}
\usepackage{xcolor}

\setcitestyle{authoryear,open={(},close={)}}
\usepackage{graphicx}
\usepackage{subcaption}
\usepackage{epstopdf}
\usepackage{float}

\newtheorem{theorem}{Theorem}%[section]
\newtheorem{Example}{Example}
\numberwithin{equation}{section}
\newcommand{\ds}{\displaystyle}
\begin{document}
\begin{frontmatter}
\title{A Note on Improved Multivariate Normal Mean Estimation With
Unknown Covariance When p Is Greater Than n%\thanksref{T1}
} \runtitle{Correction: A Note on Improved Multivariate Normal Mean Estimation}
%\thankstext{T1}{Footnote to the title with the ``thankstext'' command.}
%\thankstext{T1}{Footnote to the title with the ``thankstext'' command.}
\begin{aug}
\author{\fnms{Arash} \snm{A. Foroushani } \thanksref{t2,m2}
%\thanksref{t3,m1,m2}
\ead[label=e2]{aghaeif@uwindsor.ca}}
\and
\author{\fnms{S\'ev\'erien} \snm{Nkurunziza} \thanksref{t2,m2}
%\thanksref{t3,m1,m2}
\ead[label=e3]{severien@uwindsor.ca}}
%\author{\fnms{Third} \snm{Author}\thanksref{t1,m2}
%\ead[label=e3]{third@somewhere.com}
%\ead[label=u1,url]{http://www.foo.com}}

%\thankstext{t1}{Some comment}
\thankstext{t2}{Supported by Natural Sciences and Engineering Research Council of Canada}
%\thankstext{t3}{Natural Sciences and Engineering Research Council of Canada}
\runauthor{Arash A. Foroushani and S\'ev\'erien Nkurunziza}

\affiliation{ University of Windsor, Mathematics and Statistics department\\
	401 Sunset Avenue, Windsor, ON, N9B 3P4\\
	Email:{\color{blue}aghaeif@uwindsor.ca} E-mail:{\color{blue}severien@uwindsor.ca}\thanksmark{m2}}

%\address{Address of the Third author\\
%Usually a few lines long\\
%Usually a few lines long\\
%\printead{e3}\\
%\printead{u1}}
\end{aug}
%\title{Improved inference in tensor regression with multiple change-points}
%\author{Mai Ghannam and S\'ev\'erien Nkurunziza}
\maketitle

\begin{abstract}
In this paper, we highlight a major error in the proofs of the important results of \cite{ChetelatWells}
[D.Ch\'etelat and M. T. Wells~(2012). Improved Multivariate Normal Mean Estimation with
Unknown Covariance when $p$ is Greater than $n$. The Annals of Statistics, Vol. 40, No.~6, 3137--3160].
In particular, the proofs of some of their main results are based on Theorem~2 whose proof needs to be revisited. More precisely, there are some major mistakes in the derivation of this important result. Further, under a very realistic assumption about the rank of the estimator of the variance-covariance matrix, we correct the proof of the quoted result.
\end{abstract}
%\newpage
%\tableofcontents
%\newpage
%\begin{keyword}[class=MSC]
%\kwd[Primary ]{}
%\kwd{}
%\kwd[; secondary ]{}
%\end{keyword}

%\begin{keyword}
%\kwd{}
%\kwd{}
%\end{keyword}


\end{frontmatter}
%\tableofcontents
\section{Introduction}\label{sec:intro}
In this paper, we revisit the important results in \cite{ChetelatWells} which deals with Improved Multivariate Normal Mean Estimation with
Unknown Covariance when $p$ is Greater than $n$. Given the importance of the topic in decision theory and shrinkage methods, it is certain that the cited article will be used by a large number of researchers. Nevertheless, there is a growing interest in  statistics and thus, more often than not, one has practitioners who do not have a strong background in mathematical statistics. Because of that, a wrong mathematical result could spread very quickly.   Thus, for the benefit of the statistical science, we point-out a major error in the proofs of important results in \cite{ChetelatWells}.
More precisely, the derivations of some main results of \cite{ChetelatWells} are based on their Theorem~2. Nevertheless,  we point out a major mistake in the proof of this important result. Further, under a very realistic assumption of  about the rank of the estimator of the variance-covariance matrix, we correct the proof of the quoted result. Namely, by the incorrect use of Cauchy-Schwarz inequality, at their page 3153, \cite{ChetelatWells} obtain an incorrect bound for the quantity $X'(T^{+}TA)^{+}(T^{+}TA)X$. Indeed, in contrast with the claim in \cite{ChetelatWells},
$X'(T^{+}TA)^{+}(T^{+}TA)X\nleqslant X'(T^{+}TA)^{+}(T^{+}TA)^{+}XX'(AT^{+}T)(T^{+}TA)X$.
 Below, we give a counter-example which shows that the bound obtained in the quoted paper is incorrect.

The rest of this paper is organized as follows. \Cref{sec:mainresult} has three sections which give the main result of this paper. Specifically, \Cref{counter-examplebound} gives a counter-example which shows that the inequality used in \cite{ChetelatWells} does not hold. Further, \Cref{subsec:theorem} reformulates the Theorem~2 in \cite{ChetelatWells} and we revise, in this subsection, the proof given in the quoted paper. Finally, in \Cref{subsec:rankAss}, we give an example which shows that without imposing an assumption about the rank of the estimator of the covariance-variance matrix, an important random quantity involved in the proof of Theorem~2 of \cite{ChetelatWells} has infinite expectation. This shows the necessity to revisit the proof given in the quoted paper.
\section{Main result}\label{sec:mainresult}
\subsection{Counter-example}\label{counter-examplebound}
In this subsection, we give an example which shows that the bound obtained at page 3153 of \cite{ChetelatWells} is not correct. To this end, we use the same notations as used in \cite{ChetelatWells} and let $T$ be a symmetric matrix and $A$ a positive definite matrix. Specifically, for a given $X$ a column vector, in contrast with the statement in \cite{ChetelatWells},
%$$\frac{1}{X^{\top}(T^{+}TA)^{+}(AT^{+}T)^{+}X}\leq \frac{X^{\top}AT^{+}TAX}{X^{\top}(T^{+}TA)^{+}(T^{+}TA)X}$$
$$X^{\top}(T^{+}TA)^{+}(T^{+}TA)X \nleqslant X^{\top}(T^{+}TA)^{+}(AT^{+}T)^{+}XX^{\top}(AT^{+}T)(T^{+}TA)X.$$
\begin{Example}
    Let $A=I_4$, $X=\begin{bmatrix}
1 & 0 & 0 & 0
\end{bmatrix}^T$
and $T=\frac{1}{48}\begin{bmatrix}
7 & 7 & 1 & 1\\
7 & 7 & 1 & 1\\
1 & 1 & 7 & 7\\
1 & 1 & 7 & 7
\end{bmatrix}$. %\\
Therefore $T^{+}=\frac{1}{4}\begin{bmatrix}
7 & 7 & -1 & -1\\
7 & 7 & -1 & -1\\
-1 & -1 & 7 & 7\\
-1 & -1 & 7 & 7
\end{bmatrix}$ and then,
%\\
%Thus
\begin{flalign*}
T^{+}T=\frac{1}{192}\begin{bmatrix}
                   7 & 7 & -1 & -1\\
                    7 & 7 & -1 & -1\\
                   -1 & -1 & 7 & 7\\
                   -1 & -1 & 7 & 7
                  \end{bmatrix}
                  \begin{bmatrix}
                  7 & 7 & 1 & 1\\
                  7 & 7 & 1 & 1\\
                  1 & 1 & 7 & 7\\
                  1 & 1 & 7 & 7
                  \end{bmatrix}=\begin{bmatrix}
                  \frac{1}{2} & \frac{1}{2} & 0 & 0\\
                  \frac{1}{2} & \frac{1}{2} & 0 & 0\\
                  0 & 0 & \frac{1}{2} & \frac{1}{2}\\
                  0 & 0 & \frac{1}{2} & \frac{1}{2}
                  \end{bmatrix}.%&&
\end{flalign*}
Therefore
\begin{eqnarray*}
   T^{+}TA&=&\begin{bmatrix}
                  \frac{1}{2} & \frac{1}{2} & 0 & 0\\
                  \frac{1}{2} & \frac{1}{2} & 0 & 0\\
                  0 & 0 & \frac{1}{2} & \frac{1}{2}\\
                  0 & 0 & \frac{1}{2} & \frac{1}{2}
                  \end{bmatrix}\begin{bmatrix}
                  1 & 0 & 0 & 0\\
                  0 & 1& 0 & 0\\
                  0 & 0 & 1 & 0\\
                  0 & 0 & 0 & 1
                  \end{bmatrix}=\begin{bmatrix}
                  \frac{1}{2} & \frac{1}{2} & 0 & 0\\
                  \frac{1}{2} & \frac{1}{2} & 0 & 0\\
                  0 & 0 & \frac{1}{2} & \frac{1}{2}\\
                  0 & 0 & \frac{1}{2} & \frac{1}{2}
                  \end{bmatrix}\\%&&
%\end{flalign*}
%\begin{flalign*}
   AT^{+}T&=&\begin{bmatrix}
                  1 & 0 & 0 & 0\\
                  0 & 1& 0 & 0\\
                  0 & 0 & 1 & 0\\
                  0 & 0 & 0 & 1
                  \end{bmatrix}\begin{bmatrix}
                  \frac{1}{2} & \frac{1}{2} & 0 & 0\\
                  \frac{1}{2} & \frac{1}{2} & 0 & 0\\
                  0 & 0 & \frac{1}{2} & \frac{1}{2}\\
                  0 & 0 & \frac{1}{2} & \frac{1}{2}
                  \end{bmatrix}=\begin{bmatrix}
                  \frac{1}{2} & \frac{1}{2} & 0 & 0\\
                  \frac{1}{2} & \frac{1}{2} & 0 & 0\\
                  0 & 0 & \frac{1}{2} & \frac{1}{2}\\
                  0 & 0 & \frac{1}{2} & \frac{1}{2}
                  \end{bmatrix}\\%&&
%\end{flalign*}
%Therefore
%\begin{flalign*}
    (T^{+}TA)^{+}&=&\begin{bmatrix}
                  \frac{1}{2} & \frac{1}{2} & 0 & 0\\
                  \frac{1}{2} & \frac{1}{2} & 0 & 0\\
                  0 & 0 & \frac{1}{2} & \frac{1}{2}\\
                  0 & 0 & \frac{1}{2} & \frac{1}{2}
                  \end{bmatrix}^{+}=\begin{bmatrix}
                  \frac{1}{2} & \frac{1}{2} & 0 & 0\\
                  \frac{1}{2} & \frac{1}{2} & 0 & 0\\
                  0 & 0 & \frac{1}{2} & \frac{1}{2}\\
                  0 & 0 & \frac{1}{2} & \frac{1}{2}
                  \end{bmatrix}\\%&&
%\end{flalign*}
%\begin{flalign*}
    (AT^{+}T)^{+}&=&\begin{bmatrix}
                  \frac{1}{2} & \frac{1}{2} & 0 & 0\\
                  \frac{1}{2} & \frac{1}{2} & 0 & 0\\
                  0 & 0 & \frac{1}{2} & \frac{1}{2}\\
                  0 & 0 & \frac{1}{2} & \frac{1}{2}
                  \end{bmatrix}^{+}=\begin{bmatrix}
                  \frac{1}{2} & \frac{1}{2} & 0 & 0\\
                  \frac{1}{2} & \frac{1}{2} & 0 & 0\\
                  0 & 0 & \frac{1}{2} & \frac{1}{2}\\
                  0 & 0 & \frac{1}{2} & \frac{1}{2}
                  \end{bmatrix}.%&&
%\end{flalign*}
\end{eqnarray*}
%\newpage

Thus
\begin{flalign*}
    &(T^{+}TA)^{+}(AT^{+}T)^{+}=\begin{bmatrix}
                  \frac{1}{2} & \frac{1}{2} & 0 & 0\\
                  \frac{1}{2} & \frac{1}{2} & 0 & 0\\
                  0 & 0 & \frac{1}{2} & \frac{1}{2}\\
                  0 & 0 & \frac{1}{2} & \frac{1}{2}
                  \end{bmatrix}\begin{bmatrix}
                  \frac{1}{2} & \frac{1}{2} & 0 & 0\\
                  \frac{1}{2} & \frac{1}{2} & 0 & 0\\
                  0 & 0 & \frac{1}{2} & \frac{1}{2}\\
                  0 & 0 & \frac{1}{2} & \frac{1}{2}
                  \end{bmatrix}=\begin{bmatrix}
                  \frac{1}{2} & \frac{1}{2} & 0 & 0\\
                  \frac{1}{2} & \frac{1}{2} & 0 & 0\\
                  0 & 0 & \frac{1}{2} & \frac{1}{2}\\
                  0 & 0 & \frac{1}{2} & \frac{1}{2}
                  \end{bmatrix}\\
                  &(T^{+}TA)^{+}(T^{+}TA)=\begin{bmatrix}
                  \frac{1}{2} & \frac{1}{2} & 0 & 0\\
                  \frac{1}{2} & \frac{1}{2} & 0 & 0\\
                  0 & 0 & \frac{1}{2} & \frac{1}{2}\\
                  0 & 0 & \frac{1}{2} & \frac{1}{2}
                  \end{bmatrix}\begin{bmatrix}
                  \frac{1}{2} & \frac{1}{2} & 0 & 0\\
                  \frac{1}{2} & \frac{1}{2} & 0 & 0\\
                  0 & 0 & \frac{1}{2} & \frac{1}{2}\\
                  0 & 0 & \frac{1}{2} & \frac{1}{2}
                  \end{bmatrix}=\begin{bmatrix}
                  \frac{1}{2} & \frac{1}{2} & 0 & 0\\
                  \frac{1}{2} & \frac{1}{2} & 0 & 0\\
                  0 & 0 & \frac{1}{2} & \frac{1}{2}\\
                  0 & 0 & \frac{1}{2} & \frac{1}{2}
                  \end{bmatrix}\\
                  &(AT^{+}T)(T^{+}TA)=\begin{bmatrix}
                  \frac{1}{2} & \frac{1}{2} & 0 & 0\\
                  \frac{1}{2} & \frac{1}{2} & 0 & 0\\
                  0 & 0 & \frac{1}{2} & \frac{1}{2}\\
                  0 & 0 & \frac{1}{2} & \frac{1}{2}
                  \end{bmatrix}\begin{bmatrix}
                  \frac{1}{2} & \frac{1}{2} & 0 & 0\\
                  \frac{1}{2} & \frac{1}{2} & 0 & 0\\
                  0 & 0 & \frac{1}{2} & \frac{1}{2}\\
                  0 & 0 & \frac{1}{2} & \frac{1}{2}
                  \end{bmatrix}=\begin{bmatrix}
                  \frac{1}{2} & \frac{1}{2} & 0 & 0\\
                  \frac{1}{2} & \frac{1}{2} & 0 & 0\\
                  0 & 0 & \frac{1}{2} & \frac{1}{2}\\
                  0 & 0 & \frac{1}{2} & \frac{1}{2}
                  \end{bmatrix}.&&
\end{flalign*}
Therefore
\begin{flalign*}
    &X^{\top}(T^{+}TA)^{+}(T^{+}TA)X=\begin{bmatrix}
1 & 0 & 0 & 0
\end{bmatrix}\begin{bmatrix}
                  \frac{1}{2} & \frac{1}{2} & 0 & 0\\
                  \frac{1}{2} & \frac{1}{2} & 0 & 0\\
                  0 & 0 & \frac{1}{2} & \frac{1}{2}\\
                  0 & 0 & \frac{1}{2} & \frac{1}{2}
                  \end{bmatrix}\begin{bmatrix}
                      1\\
                      0\\
                      0\\
                      0
                  \end{bmatrix}=\begin{bmatrix}
\frac{1}{2} & \frac{1}{2} & 0 & 0
\end{bmatrix}\begin{bmatrix}
                      1\\
                      0\\
                      0\\
                      0
                  \end{bmatrix}=\frac{1}{2}.&&
\end{flalign*}
\begin{flalign*}
    &X^{\top}(T^{+}TA)^{+}(AT^{+}T)^{+}XX^{\top}(AT^{+}T)(T^{+}TA)X %\\
    %&
    =\begin{bmatrix}
1 & 0 & 0 & 0
\end{bmatrix}\begin{bmatrix}
                  \frac{1}{2} & \frac{1}{2} & 0 & 0\\
                  \frac{1}{2} & \frac{1}{2} & 0 & 0\\
                  0 & 0 & \frac{1}{2} & \frac{1}{2}\\
                  0 & 0 & \frac{1}{2} & \frac{1}{2}
                  \end{bmatrix}\begin{bmatrix}
                      1\\
                      0\\
                      0\\
                      0
                  \end{bmatrix}
                  \begin{bmatrix}
1 & 0 & 0 & 0
\end{bmatrix}\begin{bmatrix}
                  \frac{1}{2} & \frac{1}{2} & 0 & 0\\
                  \frac{1}{2} & \frac{1}{2} & 0 & 0\\
                  0 & 0 & \frac{1}{2} & \frac{1}{2}\\
                  0 & 0 & \frac{1}{2} & \frac{1}{2}
                  \end{bmatrix}\begin{bmatrix}
                      1\\
                      0\\
                      0\\
                      0
                  \end{bmatrix}\\
                  &=\begin{bmatrix}
\frac{1}{2} & \frac{1}{2} & 0 & 0
\end{bmatrix}\begin{bmatrix}
                      1\\
                      0\\
                      0\\
                      0
                  \end{bmatrix}
                  \begin{bmatrix}
1 & 0 & 0 & 0
\end{bmatrix}\begin{bmatrix}
                  \frac{1}{2} & \frac{1}{2} & 0 & 0\\
                  \frac{1}{2} & \frac{1}{2} & 0 & 0\\
                  0 & 0 & \frac{1}{2} & \frac{1}{2}\\
                  0 & 0 & \frac{1}{2} & \frac{1}{2}
                  \end{bmatrix}\begin{bmatrix}
                      1\\
                      0\\
                      0\\
                      0
                  \end{bmatrix}%\\
                  %&
                  =\frac{1}{2}\begin{bmatrix}
1 & 0 & 0 & 0
\end{bmatrix}\begin{bmatrix}
                  \frac{1}{2} & \frac{1}{2} & 0 & 0\\
                  \frac{1}{2} & \frac{1}{2} & 0 & 0\\
                  0 & 0 & \frac{1}{2} & \frac{1}{2}\\
                  0 & 0 & \frac{1}{2} & \frac{1}{2}
                  \end{bmatrix}\begin{bmatrix}
                      1\\
                      0\\
                      0\\
                      0
                  \end{bmatrix}%\\
                  %&
                  =\frac{1}{2}\begin{bmatrix}
\frac{1}{2} & \frac{1}{2} & 0 & 0
\end{bmatrix}\begin{bmatrix}
                      1\\
                      0\\
                      0\\
                      0
                  \end{bmatrix}=\frac{1}{2}(\frac{1}{2})=\frac{1}{4}.&&
\end{flalign*}
Hence
$$\frac{1}{2}=X^{\top}(T^{+}TA)^{+}(T^{+}TA)X \nleq X^{\top}(T^{+}TA)^{+}(AT^{+}T)^{+}XX^{\top}(AT^{+}T)(T^{+}TA)X=\frac{1}{4},$$
\end{Example}
this proves that the inequality used at page 3153 of \cite{ChetelatWells} is incorrect.
%\newpage
\subsection{Theorem~2 of \cite{ChetelatWells} revisited}\label{subsec:theorem}
Below, we revise Theorem~2 of \cite{ChetelatWells}. To this end, let $X\sim \mathcal{N}_p(\theta,\Sigma)$ and $Y\sim \mathcal{N}_{n\times p}(0,I_n\otimes\Sigma)$ and let $F=X^{\top}S^{+}X$. Without an assumption about the rank of the matrix $S$, some steps in proof of \cite{ChetelatWells} are wrong. Indeed, we give a counter-example which shows that if the rank of $S$ is less than or equal to 2, $\textnormal{E}\left(\ds{\frac{1}{F}}\right)=\infty$ and this contradicts an important step used in proof of Theorem~2 of \cite{ChetelatWells}.  This motivates us to also revise the statement of the theorem as well as the incorrect steps in the proof given in \cite{ChetelatWells}.
\begin{theorem}
    Let $X\sim \mathcal{N}_p(\theta,\Sigma)$ and $Y\sim \mathcal{N}_{n\times p}(0,I_n\otimes\Sigma)$ and for $A$ the symmetric positive definite square root of $\Sigma$, let $\Tilde{Y}=YA^{-1}$. Let $r$ be any bounded differentiable non-negative function $r:\mathbb{R}\longrightarrow [0,C_1]$ with bounded derivative $|r^{\prime}|\le C_2$. Define
   % \begin{flalign*}
        $G=\frac{r^2(X^{\top}S^{+}X)}{(X^{\top}S^{+}X)^2}S^{+}XX^{\top}S^{+}S$,
    %\end{flalign*}
    and $H=AGA^{-1}$. Let $S=Y^{\top}Y$, let $R=\textnormal{rank}(S)$ and suppose that  $\textnormal{P}(R>2)=1$. Then
    \begin{flalign*}
        \textnormal{E}\left[\left|\textnormal{div}_{\textnormal{\textnormal{vec}}(\Tilde{Y})}\textnormal{\textnormal{vec}}(\Tilde{Y}H)\right|\right]<\infty.
    \end{flalign*}
\end{theorem}
    \begin{proof}
     We have
     %\begin{flalign*}
         $\textnormal{div}_{\textnormal{vec}(\Tilde{Y})}\textnormal{vec}(\Tilde{Y}H) = \frac{r^2(F)}{F}(n+p-2\textnormal{tr}(SS^{+})+3)-4r(F)r^{\prime}(F)$.
     %\end{flalign*}
        Then, by the triangle inequality, we get
        \begin{flalign*}
            \left|\textnormal{div}_{\textnormal{vec}(\Tilde{Y})}\textnormal{vec}(\Tilde{Y}H)\right| & \leq \left|\frac{r^2(F)}{F}(n+p-2\textnormal{tr}(SS^{+})+3)\right|+\left|4r(F)r^{\prime}(F)\right|, %\\
 %           &\leq \frac{C^2_1}{F}|n+p-2\textnormal{tr}(SS^{+})+3|+4C_1C_2&&
        \end{flalign*}
        this gives
               \begin{flalign*}
            \left|\textnormal{div}_{\textnormal{vec}(\Tilde{Y})}\textnormal{vec}(\Tilde{Y}H)\right| \leq \frac{C^2_1}{F}\left|n+p-2\textnormal{tr}(SS^{+})+3\right|+4C_1C_2.&&
        \end{flalign*}
        Therefore, since $\textnormal{tr}(SS^{+})=\min(n,p)$, we get
        \begin{flalign*}
            \textnormal{E}[ \left|\textnormal{div}_{\textnormal{vec}(\Tilde{Y})}\textnormal{vec}(\Tilde{Y}H)\right|] \leq C^2_1\left|n+p-2min(n,p)+3\right|\textnormal{E}[1/F]+4C_1C_2&&
        \end{flalign*}
        Thus, the proof is completed if we prove that $\textnormal{E}\left[\ds{1/F}\right]<+\infty$. To this end, let $C(R)$ be $R \times p$-matrix of the form $C(R)=[I_{R}\vdots0_{R\times (p-R)}]$ and let $X_{(1)} = C(R)X$.   Let $\lambda^{\dag}_{\min}(S^{+})$ and $\lambda_{\max}^{\dag}(S^{+})$ be the smallest and biggest nonzero eigenvalues of $S^{+}$ respectively. Since $S^{+}$ is semi-positive definite,  %we have
        \begin{flalign}
           \lambda_{\min}^{\dag}(S^{+}){ X^{\top}_{(1)}X_{(1)}}\leqslant {X^{\top}S^{+}X} \leqslant \lambda_{\max}^{\dag}(S^{+}){ X^{\top}_{(1)}X_{(1)}}.\label{ineq_Fmain1}
       \end{flalign}
       Since $X\sim \mathcal{N}_{p}(\theta,\Sigma)$, we have $$X_{(1)}=C(R)X\big|R\sim \mathcal{N}_{R}(C(R)\theta,(C(R)\Sigma C^{\top}(R))).$$ Note that, since $\Sigma$ is positive definite matrix, $C(R)\Sigma C^{\top}(R)$ is positive definite matrix with probability one. Let $A(R)=\left(C(R)\Sigma C^{\top}(R)\right)^{1/2}$ and let $U=A^{-1}(R)X_{(1)}$. Then,
\begin{flalign}
   \lambda_{\min}(A^{2}(R))U^{\top}U\leqslant U^{\top}A^{2}(R)U \leqslant U^{\top}U \lambda_{\max}(A^{2}(R)),\label{ineq_Fmain2}
\end{flalign}
where $ \lambda_{\min}\left(A^{2}(R)\right)$ and $\lambda_{\max}\left(A^{2}(R)\right)$ are the smallest and biggest eigenvalues of the positive definite matrix $A(R)$ respectively. Therefore, by \eqref{ineq_Fmain1} and \eqref{ineq_Fmain2}, we get
\begin{eqnarray}
  \frac{1}{F}\leqslant \frac{1}{\lambda^{\dag}_{\min}(S^{+})\lambda_{\min}(A^{2}(R))U^{\top}U}
     =\frac{\lambda_{\max}^{\dag}(S)\lambda_{\max}(A^{-2}(R))}{U^{\top}U},\label{ineqF}
\end{eqnarray}
       where $\lambda_{\max}^{\dag}(S)$ and $\lambda_{\max}\left(A^{-2}(R)\right)$ are the biggest nonzero eigenvalues of $(S^{+})^{+}=S$ and $(A^{2}(R))^{+}=A^{-2}(R)$ respectively.
We also have
       \begin{flalign*}
          U\big|R \sim \mathcal{N}_{R}\left(A^{-1}(R)C(R)\theta,A^{-1}(R)A^{2}(R)A^{-1}(R)\right)
       \end{flalign*}
       Therefore,
       \begin{flalign}
             U\big|R \sim \mathcal{N}_{R}\left(A^{-1}(R)C(R)\theta,I_{R}\right).\label{dvecU}
       \end{flalign}
        Note that $\lambda_{\max}^{\dag}(S)$ depends on $S$, $\lambda_{\max}(A^{-2}(R))$ depends on $R$, and $U$ depends on $R$ and $X$. Since, conditionally to $R$, $S$ and $X$ are independent we get
       \begin{flalign*}
         \textnormal{E}\left[\frac{\lambda_{\max}^{\dag}(S)\lambda_{\max}(A^{-2}(R))}
         {U^{\top}U}\right] &= \textnormal{E}\left[\textnormal{E}\left[\frac{\lambda_{\max}^{\dag}(S)\lambda_{\max}(A^{-2}(R))}{U^{\top}
         U}\big|R\right]\right]\\&=\textnormal{E}\left[\lambda_{\max}(A^{-2}(R))\textnormal{E}\left[\lambda_{\max}^{\dag}(S)\big|R\right]
         \textnormal{E}\left[\frac{1}{U^{\top}U}\big|R\right]\right].&&
       \end{flalign*}
        Further, let $\lambda_{\max}(\Sigma)$ be the biggest eigenvalue of $\Sigma$. Then,
       \begin{flalign*}
           \lambda_{\max}^{\dag}(S) \leqslant \textnormal{tr}(S)=\textnormal{tr}(\Sigma^{\frac{1}{2}}\Sigma^{-\frac{1}{2}}S\Sigma^{-\frac{1}{2}}
           \Sigma^{\frac{1}{2}})
           \leqslant &\lambda_{\max}(\Sigma)\textnormal{tr}((Y\Sigma^{-\frac{1}{2}})^{\top}Y\Sigma^{-\frac{1}{2}})\\&
           =\lambda_{\max}(\Sigma)\textnormal{vec}^{\top}(Y\Sigma^{-\frac{1}{2}})\textnormal{vec}(Y\Sigma^{-\frac{1}{2}}),&&
       \end{flalign*}
       where $\textnormal{vec}(Y\Sigma^{-\frac{1}{2}}) \sim \mathcal{N}_{np}(0,I_{p} \otimes I_{n})$. Therefore,  we get
       \begin{flalign*}
\textnormal{E}_{\scriptscriptstyle \bm{\theta}}[\lambda_{\max}^{\dag}(S)] \leqslant& \lambda_{\max}(\Sigma)\textnormal{E}[\textnormal{vec}^{\top}(Y\Sigma^{-\frac{1}{2}})\textnormal{vec}(Y\Sigma^{-\frac{1}{2}})]
           \\&=\lambda_{\max}(\Sigma)\textnormal{tr}(I_{p} \otimes I_{n})=\lambda_{\max}(\Sigma)\textnormal{tr}(I_{np})=\lambda_{\max}(\Sigma)np. %\label{boundLambdamax}%&&
       \end{flalign*}
       Hence
       \begin{eqnarray}
        \textnormal{E}[\lambda_{\max}^{\dag}(S)] \leqslant \lambda_{\max}(\Sigma)np. \label{boundLambdamax}
        \end{eqnarray}
        Since $U\big|R \sim \mathcal{N}_{R}\left(A^{-1}(R)C(R)\theta,I_{R}\right)$, we get
       \begin{eqnarray}
           U^{\top}U\big|R\sim \chi^{2}_{R}(\delta_R)\label{dvecUvecU}
       \end{eqnarray}
       where $\delta_R = \left(A^{-1}(R)C(R)\theta\right)^{\top}\left(A^{-1}(R))C(R)\theta\right)$. %\\
       Let $Z$ be a random variable such that $Z\big|R\sim \textnormal{Poisson}(\delta_R/2)$. By \eqref{dvecUvecU}, we have
       \begin{flalign*}
            &\textnormal{E}\left[(U^{\top}U)^{-1}\big|R\right]=
           \textnormal{E}\left[\textnormal{E}[(U^{\top}U)^{-1}\big|R,Z]\big|R\right]%\\
           = \textnormal{E}\left[\textnormal{E}[(\chi^{2}_{R+2Z})^{-1}\big|R,Z]\big|R\right].
       \end{flalign*}
       Further, since $\textnormal{P}_{\scriptscriptstyle \bm{\theta}}(R>2)=\textnormal{P}_{\scriptscriptstyle \bm{\theta}}(R\geq 3)=1$ and $\textnormal{P}_{\scriptscriptstyle \bm{\theta}}\left(Z\geq 0\right)=1$, then $R+2Z-2\geqslant1$ with probability one and then,
       \begin{flalign*}
            &\textnormal{E}\left[(U^{\top}U)^{-1}\big|R\right]
           =\textnormal{E}\left[\frac{2^{-1}\Gamma\left(\frac{R+2Z}{2}-1\right)}{\Gamma\left(\frac{R+2Z}{2}\right)}\Big|R\right].
       \end{flalign*}
       and then, since $R+2Z-2\geq 1$ with probability one,
        \begin{eqnarray*}
           \textnormal{E}[(U^{\top}U)^{-1}\big |R]
           =\textnormal{E}\left[\ds{\frac{1}{R+2Z-2}}\big|R\right] \leqslant 1,
       \end{eqnarray*}
       almost surely. Therefore, together with \eqref{ineqF}, we get
    \begin{flalign*}
       \textnormal{E}\left[\frac{1}{F}\right]\leqslant \textnormal{E}\left[\lambda_{\max}(A^{-2}(R))\textnormal{E}\left[\lambda_{\max}^{\dag}(S)\big|R\right]\textnormal{E}
       \left[\frac{1}{U^{\top}U}\big|R\right]\right] \leqslant &\textnormal{E}\left[\lambda_{\max}(A^{-2}(R))\textnormal{E}\left[\lambda_{\max}^{\dag}(S)\big|R\right]\right]
       .&&
    \end{flalign*}
 This gives
         $\ds\textnormal{E}\left[1/F\right]\leqslant
         \textnormal{E}\left[\textnormal{tr}(A^{-2}(R))\textnormal{E}\left[\textnormal{tr}(S)\big|R\right]\right]\leqslant np\lambda_{\max}(\Sigma)\textnormal{E}\left[\textnormal{tr}(A^{-2}(R))\right]$.
    Note that $ \textnormal{P}\left(3\leqslant R\leqslant p\right)=1$ and then
    \begin{eqnarray}
    \textnormal{E}\left[\textnormal{tr}(A^{-2}(R))\right]=\sum_{j=3}^{p}\textnormal{tr}(A^{-2}(j))
    \textnormal{P}\left(R=j\right)\leqslant \sum_{j=3}^{p}\textnormal{tr}(A^{-2}(j)).\label{traceofinverseofA}
    \end{eqnarray}
    For $j=3,4,\dots,p-1$, $A^{2}(j)=[I_{j}\vdots0_{j\times (p-j)}]\Sigma[I_{j}\vdots0_{j\times (p-j)}]^{\top}$ and, we set $A^{2}(p)=\Sigma$. Thus, for $j=3,4,\dots,p$, $A^{2}(j)$ is positive definite matrix and then, $0<\textnormal{tr}(A^{-2}(j))<\infty$. Hence,
    $0<\ds\sum_{j=3}^{p}\textnormal{tr}(A^{-2}(j))<+\infty$.
Therefore, together with \eqref{boundLambdamax} and \eqref{traceofinverseofA}, we get
    \begin{flalign*}
        \textnormal{E}\left[1/F\right]\leqslant \textnormal{E}\left[\lambda_{\max}^{\dag}(S)\lambda_{\max}(A^{-2}(R))\right]
        \leqslant np\lambda_{\max}(\Sigma)\sum_{j=3}^{p}\textnormal{tr}(A^{-2}(j))<\infty.
    \end{flalign*}
    Hence,
    $\textnormal{E}\left[1/F\right]\leqslant np\lambda_{\max}(\Sigma)\ds\sum_{j=3}^{p}\textnormal{tr}(A^{-2}(j)) < \infty,$ which completes the proof.
    \end{proof}

%\newpage
\subsection{The non-existence case of the expectation of $1/F$}\label{subsec:rankAss}
In the following example, we highlight the importance of requiring the assumption about the $\textnormal{rank}(S)>2$ with probability one. In particular, we show that, when $\textnormal{P}(R \leqslant 2)>0$, it is possible to have $\textnormal{E}\left[\frac{1}{F}\right]=\infty$, what makes obsolete one of the important step in proof of Theorem~2 of \cite{ChetelatWells}.
\begin{Example}
    Let $X\sim \mathcal{N}_{2}\left(\begin{bmatrix}
    1 \\
    1
    \end{bmatrix},I_2\right)$ and $Y = \begin{bmatrix}
    U \vdots V
    \end{bmatrix}$ where $U$ and $V$ are independent random variable distributed as $\mathcal{N}\left(0,\,1\right)$ i.e.  $Y \sim \mathcal{N}_{1\times2}(0,1\otimes I_2)$. Let $R$ be the rank of $S=Y^{\top}Y$.
    Let $S^{+}=PD^{+}P^{\top}$ be the spectral decomposition of $S^{+}$ where $D^{+}=diag(d_1,0)$. Since
    \begin{flalign*}
        F=X^{\top}S^{+}X=X^{\top}PD^{+}P^{\top}X=(P^{\top}X)^{\top}D^{+}P^{\top}X.&
    \end{flalign*}
    Therefore
    \begin{flalign*}
        \frac{F}{d_1}=(P^{\top}X)^{\top}\begin{bmatrix}
            1 & 0 \\
            0 & 0
        \end{bmatrix}P^{\top}X.
    \end{flalign*}
   Note that $d_1$ and $P$ are  functions of $(U,V)$ and note that $P^{\top}X\big| U,V\sim \mathcal{N}_{2}\left(P^{\top}\begin{bmatrix}
    1 \\
    1
    \end{bmatrix},I_2\right)$. Then, % and we get
    \begin{flalign*}
      \frac{X^{\top}S^{+}X}{d_1} \big|U,V \sim \chi^2_1(\delta_{0})
    \end{flalign*}
            where $\delta_{0} = \begin{bmatrix}
        1 & 1
    \end{bmatrix}P\begin{bmatrix}
        1 & 0 \\
        0 & 0
    \end{bmatrix}P^{\top}\begin{bmatrix}
        1\\
        1
    \end{bmatrix}.$
    %\\
    Therefore, % we get
     \begin{flalign*}
        \textnormal{E}\left[\frac{d_1}{F}\big|U,V\right]= \textnormal{E}\left[\frac{d_1}{X^{\top}S^{+}X}\big|U,V\right]=\textnormal{E}\left[(\chi^2_1(\delta_{0}))^{-1}\big|U,V\right]=+\infty,&&
     \end{flalign*}
     almost surely.
     Thus
     \begin{eqnarray*}
        \textnormal{E}\left[\frac{1}{F}\big|U,V\right]=\frac{1}{d_1}\textnormal{E}\left[\frac{d_1}{F}\big|U,V\right]=+\infty,
     \end{eqnarray*}
 \mbox{ with probability one.}     Hence $$\textnormal{E}\left[\frac{1}{F}\right] = +\infty.$$
\end{Example}
\bibliographystyle{abbrvnat}
\bibliography{AoS_subm_Jul18_2023}
\end{document}